\documentclass[runningheads]{llncs}

\usepackage[T1]{fontenc}
\usepackage[utf8]{inputenc}
\usepackage{lmodern}
\usepackage{amsmath,amssymb,mathtools}
\usepackage{booktabs}
\usepackage{array}
\usepackage{tabularx}
\usepackage{microtype}
\usepackage{url}

\usepackage[main=english]{babel}
\usepackage{xcolor}
\usepackage{cases}
\usepackage[unicode,hidelinks]{hyperref}
\hypersetup{
  pdftitle={One-Dimensional Symmetry Reductions of the Monge--Ampere Equation on Explicit Convex Cones in R4},
  pdfauthor={Amir Valiakhmetov, Ewsey Obzherin, Alexey Chernov, Roland Hildebrand},
  pdfsubject={Explicit convex cones, canonical barriers, and Monge--Ampere reductions},
  pdfkeywords={conic optimization, canonical barrier, affine sphere, Monge--Ampere equation},
  bookmarksopen=true,
  bookmarksnumbered=true
}
\newcolumntype{Y}{>{\raggedright\arraybackslash}X}
\newcolumntype{L}[1]{>{\raggedright\arraybackslash}p{#1}}

\DeclareMathOperator{\Aut}{Aut}
\DeclareMathOperator{\aut}{aut}
\DeclareMathOperator{\spanop}{span}
\DeclareMathOperator{\diag}{diag}
\DeclareMathOperator{\cl}{cl}

\newcommand{\gl}{\mathfrak{gl}}

\newcommand{\phid}{\dot\phi}
\newcommand{\phidd}{\ddot\phi}

\newcommand{\bkappa}{\boldsymbol{\kappa}}

\begin{document}

\title{One-Dimensional Symmetry Reductions of the Monge--Amp\`ere Equation on some Semi-Homogeneous Convex Cones in \texorpdfstring{$\mathbb R^4$}{R4}}
\titlerunning{Symmetry Reductions on Explicit Convex Cones}

\author{Amir Valiakhmetov\inst{1} \and Ewsey Obzherin\inst{1,2}\thanks{Corresponding author.} \and Alexey Chernov\inst{1} \and Roland Hildebrand\inst{1}}
\authorrunning{A. Valiakhmetov et al.}
\institute{Moscow Institute of Physics and Technology (MIPT), Dolgoprudny, Russia
\and FusionBrain Lab\\
\email{\href{mailto:obzherin.ewsey@ya.ru}{obzherin.ewsey@ya.ru}}}

\maketitle

\begin{abstract}
Motivated by non-symmetric conic optimization, we construct six explicit families of semi-homogeneous proper convex cones $K$ in $\mathbb R^4$. For each family, a specified three-dimensional subgroup of linear automorphisms of $K$ has generic orbits of codimension one. This yields an orbit invariant $t$ and reduces the Monge-Amp\`ere equation $\det F''=e^{2F}$, defining the canonical barrier on $K$, to an ordinary differential equation (ODE) with independent variable $t$. The families include asymmetric power, rotational-power, exponential, and nonabelian-type cones. For the power and exponential cases we derive first integrals and parametric solutions of the ODE. The reductions open the way for computing the canonical barrier on these cones numerically.

\keywords{Convex cones \and Monge--Amp\`ere equation \and Conic optimization \and Canonical barrier}
\end{abstract}

\section{Introduction}

Interior-point methods for conic optimization require efficiently evaluable logarithmically homogeneous self-concordant barriers. Such barriers are explicitly known for symmetric cones \cite{NesterovNemirovski1994}. Nonsymmetric examples of cones used in optimization are power and exponential cones \cite{Nesterov12,Alfonso22,Hypathia22}. Generalized power cones also have useful modeling and structural properties \cite{LinLindstromLourencoPong2024}. All these cones are \textit{semi-homogeneous}, i.e., their automorphism group $\Aut(K)$ has generic orbits of at most codimension 1 \cite{Hildebrand2014Analytic}. 

For every proper convex cone $K\subset\mathbb R^n$, the canonical barrier \cite{Hildebrand2014Canonical,Fox2012} is the unique solution $F: K^{\circ} \to \mathbb R$ of the Monge-Amp\`ere equation
\begin{equation}\label{eq:ma}
    \det F''=e^{2F},\qquad F''\succ0,\qquad F|_{\partial K}=+\infty.
\end{equation}
It satisfies the logarithmic homogeneity condition
\begin{equation}\label{eq:homogeneous}
        F(\lambda x)=F(x)-n\log\lambda ,\qquad \lambda>0,\ x \in K^{\circ}.
\end{equation}
Here the level sets of $F$ are the complete hyperbolic affine spheres asymptotic to $\partial K$, and existence and uniqueness follow from corresponding work\cite{Calabi1972,ChengYau1980,ChengYau1982,Sasaki1985,ChengYau1986,NomizuSasaki1994,LiSimonZhao1993}. 

In this paper we consider explicit families of 4-dimensional semi-homogeneous cones. We designate an explicit connected subgroup $G\subset\Aut(K)$ with generic codimension 1 orbits. All orbit statements refer to $G$ unless $\Aut(K)$ is explicitly mentioned. For special cases inside each family $\Aut(K)$ may be transitive on the interior $K^{\circ}$ of the cone, but the subgroup $G$ still has orbits of smaller dimension.

We construct six explicit cone families in $\mathbb R^4$, including new nonabelian-type examples, and give full set of generators for $G$. The rotational-power family consists of generalized power cones already used in conic optimization. For this family, our contribution is an explicit reduction of \eqref{eq:ma} to an ODE. The asymmetric-power and exponential reductions generalize Hildebrand's three-dimensional power and exponential cases, whereas the two nonabelian-type families and their reduced boundary-value problems are new. For each family an orbit invariant $t: K^{\circ} \to \mathbb R$ and an explicit summand $F_0$ satisfying the relation
\begin{equation}\label{eq:auttransform}
        F(gx)=F(x)-\log|\det g|,\qquad g\in\Aut(K).
\end{equation}
are given such that the canonical barrier can be written as
\begin{equation}\label{eq:ansatz}
        F=F_0+\phi(t)
\end{equation}
with the function $\phi$ satisfying the ODE which is obtained by reducing \eqref{eq:ma}, and $F_0$ being logarithmically homogeneous and constant on each orbit of the unimodular subgroup of $G$. For the power and exponential cones we construct first integrals and parametric solutions. The nonabelian-type cases (here nonabelian refers to the Lie algebra of $G$) yield explicit boundary-value problems. Positivity, smooth matching, and asymptotic behaviour at the boundary select candidate solutions, but uniqueness of the integration constants is asserted only in the former cases. 

\section{Preliminaries} \label{sec:Normalizations}

A cone $K\subset\mathbb R^4$ is called proper if it is closed, pointed, convex, and has non-empty interior. Here pointed means $K\cap(-K)=\{0\}$.
Its automorphism group and Lie algebra are defined by
\begin{equation*}
        \Aut(K)=\{A\in GL_4(\mathbb R):\ A[K]=K\},
\end{equation*}
\begin{equation*}
        \aut(K)=\{X\in\gl_4(\mathbb R):\exp(tX)[K]=K\text{ for all }t\in\mathbb R\}.
\end{equation*}
The canonical barrier satisfies the transformation law \eqref{eq:auttransform} \cite{Hildebrand2014Canonical}.
In particular, every unimodular automorphism preserves $F$. For the one-parameter subgroups $G \subset \Aut(K)$ used below we have $\det g>0$, so the absolute value may be omitted. 
Throughout the paper, $G$ denotes the connected subgroup generated by the specified Lie algebra for each of the cone families. Statements about generic orbit dimension or codimension refer to $G$ unless $\Aut(K)$ is mentioned explicitly.

\section{Explicit Cone Families} \label{sec:MainStatement}

Throughout the paper, points in $\mathbb R^4$ are written as column vectors
$(x,y,z,w)^T$, and $\cl Q$ denotes the closure of a subset $Q \subset
\mathbb R^4$. For the two nonabelian-type families we use
\begin{equation*}
        \Delta=2xz-y^2.
\end{equation*}
We consider the following six explicit families of 4D semi-homogeneous cones. The first 4 cases lead to abelian Lie algebras, the last two are of nonabelian type.

\begin{enumerate}
\item The asymmetric power cone
\begin{equation}\label{eq:powerlayer}
\begin{split}
K_{\alpha,\beta}^{p,q,r}=\cl\{(x,y,z,w)^T:\;&x,y,z>0,\\
&-\alpha x^{1/p}y^{1/q}z^{1/r}<w
 <\beta x^{1/p}y^{1/q}z^{1/r}\},
\end{split}
\end{equation}
where $\alpha,\beta>0$, $p,q,r>1$, and $p^{-1}+q^{-1}+r^{-1}=1$. These assumptions exclude the homogeneous degenerations. E.g., $p=q=2$, $r=\infty$ and $\alpha=\beta$ leads to $K \simeq \mathcal L_3\times\mathbb R_+$, while $q = r = \infty$ yields $K \simeq \mathbb R_+^4$.

\item The rotational-power cone
\begin{equation}\label{eq:normpower}
K_{\lambda}^{2,2}=\cl\left\{(x,y,z,w)^T:
x,y>0,\ \sqrt{z^2+w^2}<x^\lambda y^{1-\lambda}\right\},
\end{equation}
where $\alpha,\beta>0$ and $0<\lambda<1$. The exceptional value $\lambda=\tfrac12$ yields $\mathcal L_4$: with $u=(x+y)/2$, and $v=(x-y)/2$, the inequality becomes $u^2>v^2+z^2+w^2$. Further, $\lambda=0,1$ gives $\mathcal L_3\times\mathbb R_+$. In the other cases we get non-homogeneous cones \cite{LinLindstromLourencoPong2024}.

\item The exponential-power cone
\begin{equation}\label{eq:exppower}
K_{{\rm pow},a}=\cl\left\{(x,y,z,w)^T:
z,w>0,\ \frac yz>\exp(x/z)\left(\frac wz\right)^{-a}\right\},
\qquad a>0.
\end{equation}

\item The exponential-quadratic cone
\begin{equation}\label{eq:exppar}
K_{{\rm par}}=\cl\left\{(x,y,z,w)^T:
z>0,\ \frac yz>\exp\left(\frac xz+\left(\frac wz\right)^2\right)\right\}.
\end{equation}

\item The asymmetric mixed-power cone
\begin{equation}\label{eq:lorentzpower}
\begin{split}
K^N_{\gamma,\alpha,\beta}=\cl\{(x,y,z,w)^T:\;&x,z>0,\ \Delta>0,\\
&-\alpha\Delta^{\gamma/2}z^{1-\gamma}<w
 <\beta\Delta^{\gamma/2}z^{1-\gamma}\},
\end{split}
\end{equation}
where $\alpha,\beta>0$ and $0<\gamma<2$.

\item The logarithmic-parabolic cone
\begin{equation}\label{eq:lorentzlog}
K^N_{\log}=\cl\left\{(x,y,z,w)^T:
x,z>0,\ \Delta>0,\ w<z\log\frac{\Delta}{z^2}\right\}.
\end{equation}
\end{enumerate}

All inequalities inside the closure describe the interiors of the respective cones. Their convexity, pointedness, and explicit automorphisms are verified in the next section. Those coordinate variables which are explicitly designated as positive will be referred to as \textit{base coordinates} further on.

\section{Convexity and Explicit Automorphisms} \label{sec:AutomorphismAlgebras}

We now define the subgroup $G \subset \Aut(K)$ by specifying a basis of its Lie algebra for every cone family. For a displayed Lie algebra element $A \in \mathbb R^{4 \times 4}$, the one-parameter group $\exp(sA)$ acts on the
column vector $(x,y,z,w)^T$. The scalar generator
\begin{equation*}
S=\left(\begin{smallmatrix}
1&0&0&0\\0&1&0&0\\0&0&1&0\\0&0&0&1
\end{smallmatrix}\right)
\end{equation*}
preserves every cone and is always an element of the Lie algebra of $G$. We now either give two more generators, or three generators
whose sum is $S$, for each family. We use only these explicit subgroups and do not claim that the displayed generators exhaust the full automorphism algebra.

\paragraph{Asymmetric power cone.}
Let $h=x^{1/p}y^{1/q}z^{1/r}$. The infinitesimal automorphisms of $K$ include
\begin{equation*}
\begin{aligned}
A_x&=\left(\begin{smallmatrix}1&0&0&0\\0&0&0&0\\0&0&0&0\\0&0&0&1/p\end{smallmatrix}\right),&
A_y&=\left(\begin{smallmatrix}0&0&0&0\\0&1&0&0\\0&0&0&0\\0&0&0&1/q\end{smallmatrix}\right),&
A_z&=\left(\begin{smallmatrix}0&0&0&0\\0&0&0&0\\0&0&1&0\\0&0&0&1/r\end{smallmatrix}\right).
\end{aligned}
\end{equation*}
Indeed, each positive base-coordinate scaling multiplies $w$ by the corresponding
factor $h$, and $A_x+A_y+A_z=S$. Since the exponents are positive and sum
to one, $h$ is concave and homogeneous of degree one on $\mathbb R_+^3$.
Both inequalities $w\leq\beta h$ and $-w\leq\alpha h$ therefore define convex
hypographs, proving convexity of \eqref{eq:powerlayer}.

\paragraph{Rotational-power cone.}
The following three matrices generate base scalings and rotations of the elliptic fibers, respectively:
\begin{equation*}
\begin{aligned}
A_x&=\left(\begin{smallmatrix}1&0&0&0\\0&0&0&0\\0&0&\lambda&0\\0&0&0&\lambda\end{smallmatrix}\right),&
A_y&=\left(\begin{smallmatrix}0&0&0&0\\0&1&0&0\\0&0&1-\lambda&0\\0&0&0&1-\lambda\end{smallmatrix}\right),\\[-1mm]
R&=\left(\begin{smallmatrix}0&0&0&0\\0&0&0&0\\0&0&0&-1\\0&0&1&0\end{smallmatrix}\right).
\end{aligned}
\end{equation*}
Here $A_x+A_y=S$, while $\exp(sR)$ preserves $z^2+w^2$.
The function $x^\lambda y^{1-\lambda}$ is concave, and
$x^\lambda y^{1-\lambda}-\sqrt{ z^2+ w^2}$ is concave. Its
nonnegative superlevel set is exactly \eqref{eq:normpower}, hence this cone is
convex.

\paragraph{Exponential-power cone.}
Besides $S$, define the Lie algebra generators
\begin{equation*}
B=\left(\begin{smallmatrix}0&0&1&0\\0&1&0&0\\0&0&0&0\\0&0&0&0\end{smallmatrix}\right),
\qquad
C=\left(\begin{smallmatrix}0&0&0&0\\0&-a&0&0\\0&0&0&0\\0&0&0&1\end{smallmatrix}\right).
\end{equation*}
Their actions are
\begin{equation*}
\begin{aligned}
\exp(sB)(x,y,z,w)^T&=(x+sz,e^sy,z,w)^T,\\
\exp(sC)(x,y,z,w)^T&=(x,e^{-as}y,z,e^sw)^T,
\end{aligned}
\end{equation*}
so the defining inequality in \eqref{eq:exppower} is invariant. Moreover,
$f(x,w)=e^xw^{-a}$ has
\begin{equation*}
\nabla^2f=f \cdot \left(\begin{smallmatrix}1&-a/w\\-a/w&a(a+1)/w^2\end{smallmatrix}\right),
\qquad \det\nabla^2f=a f^2/w^2>0.
\end{equation*}
Thus $f$ is convex, its perspective $z f(x/z,w/z)$ is convex, and
\eqref{eq:exppower} is its closed epigraph in the coordinate $y$.

\paragraph{Exponential-quadratic cone.}
Here the basis elements are $S$ and
\begin{equation*}
B=\left(\begin{smallmatrix}0&0&1&0\\0&1&0&0\\0&0&0&0\\0&0&0&0\end{smallmatrix}\right),
\qquad
C=\left(\begin{smallmatrix}0&0&0&-2\\0&0&0&0\\0&0&0&0\\0&0&1&0\end{smallmatrix}\right).
\end{equation*}
The action of $B$ is the same as above, whereas
\begin{equation*}
\exp(sC)(x,y,z,w)^T=(x-2sw-s^2z,y,z,w+sz)^T.
\end{equation*}
It leaves $x/z+(w/z)^2$ unchanged. The Hessian of
$f(x,w)=e^{x+w^2}$ is
\begin{equation*}
\nabla^2f=f\left(\begin{smallmatrix}1&2w\\2w&4w^2+2\end{smallmatrix}\right),
\qquad \det\nabla^2f=2f^2>0.
\end{equation*}
Consequently the perspective $z f(x/z,w/z)$ is convex and
\eqref{eq:exppar} is a convex epigraph.

\paragraph{Asymmetric mixed-power cone.}
Besides $S$, explicit generators are
\begin{equation*}
D=\left(\begin{smallmatrix}2&0&0&0\\0&1&0&0\\0&0&0&0\\0&0&0&\gamma\end{smallmatrix}\right),
\qquad
N=\left(\begin{smallmatrix}0&1&0&0\\0&0&1&0\\0&0&0&0\\0&0&0&0\end{smallmatrix}\right).
\end{equation*}
The action of $N$ is
\begin{equation*}
(x,y,z,w)^T\mapsto(x+sy+\tfrac12s^2z,y+sz,z,w)^T,
\end{equation*}
which fixes $\Delta$ and $z$. The action of $D$ multiplies $\Delta$ by
$e^{2s}$ and $w$ by $e^{\gamma s}$, and therefore preserves
$w/(\Delta^{\gamma/2}z^{1-\gamma})$. The function
$\sqrt\Delta$ is concave for positive base coordinates. Since the weighted geometric mean is concave and coordinate-wise increasing on $\mathbb R_+^2$,
$\Delta^{\gamma/2}z^{1-\gamma}$ is concave for $0\leq\gamma\leq2$.
The two hypograph inequalities in \eqref{eq:lorentzpower} are consequently convex.

\paragraph{Logarithmic-parabolic cone.}
Together with $S$, use
\begin{equation*}
N=\left(\begin{smallmatrix}0&1&0&0\\0&0&1&0\\0&0&0&0\\0&0&0&0\end{smallmatrix}\right),
\qquad
L=\left(\begin{smallmatrix}2&0&0&0\\0&1&0&0\\0&0&0&0\\0&0&2&0\end{smallmatrix}\right).
\end{equation*}
The first basis element again fixes $\Delta,z,w$. The second one gives
\begin{equation*}
(x,y,z,w)^T\mapsto(e^{2s}x,e^sy,z,w+2sz)^T,
\end{equation*}
which preserves $w-z\log(\Delta/z^2)$. Finally, the defining inequality is
equivalent on $z>0$ to
\begin{equation*}
        x>\frac{y^2}{2z}+\frac z2\exp(w/z).
\end{equation*}
The right-hand side is a sum of two convex perspective functions, so its
epigraph is convex.

The closures above have nonempty interiors. Pointedness is immediate for the power and rotational-power cones from the nonnegativity of their base coordinates.
For the exponential-power cone, the face $z=0$ has
$x\leq0$, $y\geq0$, and $w\geq0$; for the exponential-quadratic cone it also
has $w=0$. Hence neither cone contains a line. The base of the asymmetric mixed-power cone is pointed, so
this cone is pointed. In the logarithmic case, a vector and its
negative can both lie in the closure only if $x=y=z=0$ and simultaneously
$w\leq0$ and $-w\leq0$, hence $w=0$. Thus all six sets are proper convex
cones.

\begin{lemma}[Generic orbit rank]\label{lem:generic-orbit-rank}
For every family, the infinitesimal action of $G$ has rank three on an open dense subset of $K^\circ$, so its generic orbits have codimension one. If $t: K^{\circ} \to \mathbb R$ is an orbit invariant and $dt_x\ne0$, then
$T_x(G\cdot x)=\ker dt_x$. Hence every smooth $G$-invariant function locally factors through $t$.
\end{lemma}
\begin{proof}
For each cone family, construct a $4 \times 3$ matrix whose columns are $A_kx$, where $A_k$ is a Lie algebra generator and $x$ runs through $K^{\circ}$. We shall show that this matrix is generically of full rank. Selected $3\times3$ minors for the six families are, respectively,
$xyz$,
$-xyw$ or $xyz$,
$-wz^2$, $-z^3$, $z\Delta$, and $-z\Delta$.
They are nonzero in the interior, except on the proper closed rotation axis
$\{z=w=0\}$ in the rotational-power case. Thus the generic rank is three and so is the generic dimension of the $G$-orbits.
Invariance gives $T_x(G\cdot x)\subseteq\ker dt_x$. If $dt_x\ne0$, then both spaces have dimension three. The local factorization follows from the submersion theorem.
\end{proof}

\section{Reduction of the Monge--Amp\`ere Equation} \label{sec:MAreduction}

Let $t$ be a complete local invariant of
the generic $G$-orbits, and let $F_0: K^{\circ} \to \mathbb R$ satisfy \eqref{eq:auttransform}. Then the difference $\phi = F-F_0$ is $G$-invariant and hence $F=F_0+\phi(t)$ locally by Lemma \ref{lem:generic-orbit-rank}. The chain rule gives
\begin{equation}\label{eq:common-hessian-reduction}
 F''=F_0''+\ddot\phi\,\nabla t\nabla t^\top+\dot\phi\,\nabla^2t.
\end{equation}
Equivariance then reduces PDE \eqref{eq:ma} to an ODE on $\phi(t)$ by the determinant calculations
displayed below for each case. Dots denote $t$-derivatives, and the endpoints of the domain of $t$ provide the boundary conditions. Below we shall specify in each case an explicit summand $F_0$ and deduce the reduction for the corresponding $\phi$. Below, $c$ and $C$ both denote real integration constants.

\subsection{Asymmetric power cone}\label{subsec:power-layer-solution}

Set
\begin{equation*}
h=x^{1/p}y^{1/q}z^{1/r},\qquad t=\frac wh,
\end{equation*}
so that $-\alpha<t<\beta$. Invariance and logarithmic homogeneity give
\begin{equation}\label{eq:power-short-ansatz}
F=-\frac{p+1}{p}\log x-\frac{q+1}{q}\log y
-\frac{r+1}{r}\log z+\phi(t).
\end{equation}
For $t>0$ put
\begin{equation*}
\tau=\log t,\qquad \varphi=\phi+\tau,\qquad
\xi=\frac{d\varphi}{d\tau}=1+t\dot\phi,
\end{equation*}
and define
\begin{equation*}
A=pqr,\qquad P(\xi)=(\xi-1)(\xi+p)(\xi+q)(\xi+r).
\end{equation*}
Differentiation of \eqref{eq:power-short-ansatz} by virtue of \eqref{eq:common-hessian-reduction}, followed by the
matrix determinant lemma, reduces \eqref{eq:ma} to
\begin{equation}\label{eq:power-short-ode}
\frac{P'(\xi)}{\xi}\frac{d\xi}{d\tau}-P(\xi)=Ae^{2\varphi}.
\end{equation}
Since $d\varphi/d\tau=\xi$, integration gives
\begin{equation}\label{eq:power-short-first}
e^{-\varphi}P(\xi)=Ae^\varphi+c,
\end{equation}
where $c\in\mathbb R$. Thus, with
\begin{equation*}
D(\xi)=c^2+4AP(\xi),\qquad
H(\xi)=\frac1{\xi-1}-\frac1{p(\xi+p)}
-\frac1{q(\xi+q)}-\frac1{r(\xi+r)},
\end{equation*}
every local nonconstant branch satisfying the algebraic admissibility conditions has the parametrization
\begin{equation}\label{eq:power-short-solution}
\begin{aligned}
t(\xi)&=t_0\exp\!\left\{\int_{\xi_0}^{\xi}
\frac{H(s)}2\left(1+\frac{c\sigma}{\sqrt{D(s)}}\right)ds\right\},\\
\phi(\xi)&=\log\!\left(
\frac{-c+\sigma\sqrt{D(\xi)}}{2A\,t(\xi)}\right),
\qquad \sigma\in\{-1,1\}.
\end{aligned}
\end{equation}
Here $t_0>0$ and $\xi_0$ fix a reference point of the parametrization.
One takes components on which $D>0$, the logarithm is defined, and the Hessian
of \eqref{eq:power-short-ansatz} is positive definite. The formula on $t<0$ is obtained by reflection. The solution defining the canonical barrier must also satisfy smooth matching at $t=0$, Hessian positivity, and
$\phi(t)\to+\infty$ as $t\downarrow-\alpha$ and $t\uparrow\beta$. These conditions constrain the constants, but a general existence-and-uniqueness theorem for the tuple $(c,t_0,\xi_0,\sigma)$ is not claimed here. After normalizing $\beta=1$ and, if necessary, reflecting $t$, we may without loss of generality restrict to $-2A<c\leq0$. The limit case $c=0$ gives the symmetric power cones with $\alpha = \beta$.

For completeness, in the symmetric case $\alpha=\beta=1$ the final expression
is elementary. With $\zeta\in\mathbb R$,
\begin{equation}\label{eq:power-short-symmetric}
\begin{aligned}
t(\zeta)&=\zeta
(\zeta^2+p+1)^{-1/(2p)}
(\zeta^2+q+1)^{-1/(2q)}
(\zeta^2+r+1)^{-1/(2r)},\\
\phi(\zeta)&=-\tfrac12\log(pqr)
+\frac{p+1}{2p}\log(\zeta^2+p+1)\\
&\quad+\frac{q+1}{2q}\log(\zeta^2+q+1)
+\frac{r+1}{2r}\log(\zeta^2+r+1).
\end{aligned}
\end{equation}

\subsection{Rotational-power cone}

A generalization of this case to arbitrary dimension was considered in \cite[Section 6]{Hildebrand2014Canonical} with a different but isomorphic definition of the cones. Set
\begin{equation*}
\rho=\sqrt{z^2+w^2},\qquad
t=\frac{\rho}{x^\lambda y^{1-\lambda}}\in[0,1),
\end{equation*}
\begin{equation}\label{eq:norm-short-ansatz}
F=-(1+2\lambda)\log x-(3-2\lambda)\log y+\phi(t).
\end{equation}
At the normal section $x=y=1$, $z=t$, $w=0$, define
\begin{equation*}
a_x=1+2\lambda+\lambda t\dot\phi,\qquad
a_y=3-2\lambda+(1-\lambda)t\dot\phi,\qquad
Q=\frac{\lambda^2}{a_x}+\frac{(1-\lambda)^2}{a_y}.
\end{equation*}
The determinant of the Hessian gives the scalar ODE
\begin{equation}\label{eq:norm-short-ode}
\frac{\dot\phi}{t}\,a_xa_y
\left[\ddot\phi-\dot\phi(\dot\phi+t\ddot\phi)Q\right]=e^{2\phi}.
\end{equation}
The formula at $t=0$ is understood by its smooth limit. To integrate it, put
\begin{equation*}
\tau=\log t,\qquad \psi=\phi+2\tau,\qquad
\xi=2+t\dot\phi,
\end{equation*}
\begin{equation*}
P_\lambda(\xi)=(\xi-2)^2(1+\lambda\xi)
(1+(1-\lambda)\xi).
\end{equation*}
Equation \eqref{eq:norm-short-ode} has the first integral
\begin{equation}\label{eq:norm-short-first}
e^\psi-P_\lambda(\xi)e^{-\psi}=C.
\end{equation}
Here $C\in\mathbb R$. Let
\begin{equation*}
S_\lambda(\xi)=\sqrt{C^2+4P_\lambda(\xi)},\qquad
U_{\lambda,\sigma}(\xi)=\frac{C+\sigma S_\lambda(\xi)}2,
\qquad \sigma\in\{-1,1\}.
\end{equation*}
On every local branch with $U_{\lambda,\sigma}>0$ and positive-definite Hessian, a parametric candidate is
\begin{equation}\label{eq:norm-short-solution}
\begin{aligned}
\tau(\xi)&=\tau_0+\int_{\xi_0}^{\xi}
\frac{P_\lambda'(s)}{s\,\sigma S_\lambda(s)
U_{\lambda,\sigma}(s)}\,ds,\\
t(\xi)&=e^{\tau(\xi)},\qquad
\phi(\xi)=\log U_{\lambda,\sigma}(\xi)-2\tau(\xi).
\end{aligned}
\end{equation}
The constants $\tau_0\in\mathbb R$ and $\xi_0$ fix the reference point. A canonical branch must satisfy $F''\succ0$, smoothness at the rotation axis ($\phi(0)$ finite and $\dot\phi(0)=0$), and $\phi(t)\to+\infty$ as $t\uparrow1$. We do not prove here that these conditions determine a unique tuple $(C,\tau_0,\xi_0,\sigma)$. 

\subsection{Exponential-quadratic cone}

The cone is given by
\[
K = \{(x, y, 0, 0) \mid x \leq 0, y \geq 0\} \cup \left\{(x, y, z, w) \mid z > 0, \frac{y}{z} \geq \exp\left(\frac{x}{z} + \frac{w^2}{z^2}\right)\right\}.
\]

\noindent Invariance under the unimodular part of $G$ gives the ansatz $F(x,y,z,w)=-\log y-3\log z+\phi(t)$, where $t=\log\frac yz-\frac xz-\left(\frac wz\right)^2$.

The derivatives of $F$ are given by

\[
F' = 
\begin{pmatrix}
-\dfrac{\dot\phi}{z}\\[2mm]
\dfrac{\dot\phi-1}{y}\\[2mm]
-\dfrac{3+\dot\phi}{z} + \dfrac{x\dot\phi}{z^2} + \dfrac{2w^2\dot\phi}{z^3}\\[2mm]
-\dfrac{2w\dot\phi}{z^2}
\end{pmatrix}.
\]

The full Hessian, with no omitted symmetric entries, is
\[
\resizebox{\textwidth}{!}{$\displaystyle
F''=\begin{pmatrix}
\dfrac{\ddot\phi}{z^2}
&-\dfrac{\ddot\phi}{yz}
&\dfrac{\dot\phi+\ddot\phi}{z^2}-\dfrac{x\ddot\phi}{z^3}-\dfrac{2w^2\ddot\phi}{z^4}
&\dfrac{2w\ddot\phi}{z^3}\\[3mm]
-\dfrac{\ddot\phi}{yz}
&\dfrac{\ddot\phi-\dot\phi+1}{y^2}
&-\dfrac{\ddot\phi}{yz}+\dfrac{x\ddot\phi}{yz^2}+\dfrac{2w^2\ddot\phi}{yz^3}
&-\dfrac{2w\ddot\phi}{yz^2}\\[3mm]
\dfrac{\dot\phi+\ddot\phi}{z^2}-\dfrac{x\ddot\phi}{z^3}-\dfrac{2w^2\ddot\phi}{z^4}
&-\dfrac{\ddot\phi}{yz}+\dfrac{x\ddot\phi}{yz^2}+\dfrac{2w^2\ddot\phi}{yz^3}
&\dfrac3{z^2}+\dot\phi\!\left(\dfrac1{z^2}-\dfrac{2x}{z^3}-\dfrac{6w^2}{z^4}\right)+\ddot\phi\!\left(-\dfrac1z+\dfrac{x}{z^2}+\dfrac{2w^2}{z^3}\right)^2
&\dfrac{2w(\ddot\phi+2\dot\phi)}{z^3}-\dfrac{2xw\ddot\phi}{z^4}-\dfrac{4w^3\ddot\phi}{z^5}\\[3mm]
\dfrac{2w\ddot\phi}{z^3}
&-\dfrac{2w\ddot\phi}{yz^2}
&\dfrac{2w(\ddot\phi+2\dot\phi)}{z^3}-\dfrac{2xw\ddot\phi}{z^4}-\dfrac{4w^3\ddot\phi}{z^5}
&\dfrac{4w^2\ddot\phi}{z^4}-\dfrac{2\dot\phi}{z^2}
\end{pmatrix}. $}
\]
Direct determinant computation gives

\[
\det F'' = \dfrac{2\phid\left(-\phid^3+\phid^2+4\phid\phidd-3\phidd\right)}{y^2z^6}
\]

Consequently, the reduced ODE is

\begin{equation} \label{exppar_ode}
2\phid\left(-\phid^3+\phid^2+4\phid\phidd-3\phidd\right) = e^{2\phi}
\end{equation}

By Sylvester's criterion (using the principal minor indexed by the first,
second, and fourth coordinates), positive definiteness is equivalent to the
following system:

\begin{equation}\label{eq:exppar-positivity}
\left\{
\begin{aligned}
\ddot\phi&>0,\\
\dot\phi&<1,\\
2\dot\phi\ddot\phi(\dot\phi-1)&>0,\\
\ddot\phi&>\frac{\dot\phi^2(1-\dot\phi)}{3-4\dot\phi}.
\end{aligned}
\right.
\end{equation}

\noindent Let $p=\dot\phi$ and $P(p)=2p^3(1-p)$. Then the ODE takes the form 
\[
\dfrac{d}{d\phi}(e^{-\phi}P(p)) = -e^{\phi}
\]

As in \cite[Sec.~4.1]{Hildebrand2014Analytic}, this yields
$-P + \dfrac{c^2}{4} = (e^{\phi} - \dfrac{c}{2})^2$. The inequalities in
\eqref{eq:exppar-positivity} imply $P < 0$. Also, $e^{\phi} > 0$, so

\begin{equation} \label{exppar_e^phi}
e^{\phi} = \dfrac{c}{2} + \sqrt{-P + \dfrac{c^2}{4}}
\end{equation}

\noindent Let $\bkappa = -\dfrac{1}{\dot\phi}$. From \eqref{exppar_ode} we have

\[
\dfrac{dt}{d\dot\phi} = \dfrac{1}{\ddot\phi} = \dfrac{8\dot\phi^2 - 6\dot\phi}{2\sqrt{-P + \dfrac{c^2}{4}}\left(\dfrac{c}{2}+\sqrt{-P+\dfrac{c^2}{4}}\right)}
\]

\noindent This integrates to

\[
t = \int_{0}^{\bkappa} \dfrac{8 + 6\bkappa}{2\sqrt{2 + 2\bkappa + \dfrac{c^2}{4}\bkappa^4}\left(\dfrac{c}{2}\bkappa^2 + \sqrt{2 + 2\bkappa + \dfrac{c^2}{4}\bkappa^4}\right)}d\bkappa
\]

As in \cite[Sec.~4.1]{Hildebrand2014Analytic}, the affine half-line
$\left\{(x,1,1,0)\mid x\leq0\right\}\subset K$ intersects all level surfaces
of $F$. Hence $F(x,1,1,0)\to-\infty$ as $x\to-\infty$, which gives
$\phi(t)\to-\infty$ and $e^{\phi(t)}\to0$ as $t\to+\infty$. It follows from
\eqref{exppar_e^phi} that $c\leq0$.

Let us consider two cases: in the first, we arrive at a contradiction, while in the second, we obtain a parameterized expression for the function.

\textbf{Case 1:} $c<0$. The integrand above has the asymptotic expansion
\[
\dfrac{8 + 6\bkappa}{2\sqrt{2 + 2\bkappa + \dfrac{c^2}{4}\bkappa^4}\left(\dfrac{c}{2}\bkappa^2 + \sqrt{2 + 2\bkappa + \dfrac{c^2}{4}\bkappa^4}\right)}
=3+\dfrac{1}{\bkappa}+O(\bkappa^{-2}),\qquad \bkappa\to+\infty.
\]
It follows that
\[
        t=3\bkappa+\log\bkappa+O(1),
        \qquad
        \bkappa=\frac t3-\frac13\log t+O(1).
\]
Using \eqref{exppar_e^phi},
\[
\begin{aligned}
\phi
&=\log\left(\frac c2+\sqrt{\frac{c^2}{4}+2\bkappa^{-3}+2\bkappa^{-4}}\right)  \\
&=\log\left(-\frac{2}{c}\right)-3\log\bkappa+O(\bkappa^{-1})  \\
&=\log\left(-\frac{54}{c}\right)-3\log t
        +O\left(\frac{\log t}{t}\right).
\end{aligned}
\]
Therefore
\[
\begin{aligned}
\lim_{z\to0^+}F(-1,1,z,0)
&=\lim_{z\to0^+}\left(-3\log z+
\phi\left(\frac1z-\log z\right)\right) \\
&=\lim_{z\to0^+}\left(-3\log z+
\log\left(-\frac{54}{c}\right)
-3\log\left(z^{-1}-\log z\right)\right) \\
&=\log\left(-\frac{54}{c}\right).
\end{aligned}
\]
This contradicts the boundary condition $F|_{\partial K}=+\infty$. Hence the complete solution cannot correspond to $c<0$.

\textbf{Case 2:} $c=0$. We get
\[
\begin{aligned}
t
&=\int_{-\infty}^{\dot\phi}
\frac{4s^2-3s}{2(s^4-s^3)}\,ds
=\int_{-\infty}^{\dot\phi}
\left(\frac{3}{2s^2}-\frac{1}{2s}+\frac{1}{2(s-1)}\right)ds  \\
&=-\frac{3}{2\dot\phi}+\frac12\log\frac{\dot\phi-1}{\dot\phi}.
\end{aligned}
\]
And now from \eqref{exppar_e^phi} we obtain
\[
        \phi=\frac12\log\left(2(\dot\phi^4-\dot\phi^3)\right).
\]
It yields the parametric representation for the solution $\phi(t)$, where $\bkappa$ runs through all positive reals:
\[
\begin{pmatrix}
t \\
\phi
\end{pmatrix}
=
\frac12
\begin{pmatrix}
\log(1+\bkappa)+3\bkappa \\
\log(1+\bkappa)-4\log\bkappa+\log 2
\end{pmatrix}.
\]

\subsection{Exponential-power cones}

The exponential-power family is defined as follows:

\[
\begin{aligned}
K = &\{(x, y, 0, w) \mid x \leq 0, y \geq 0, w \geq 0\} \\ \cup &\left\{(x, y, z, w) \mid z > 0, w > 0, \frac{y}{z} \geq \exp\left(\frac{x}{z}\right) \left(\frac{w}{z}\right)^{-a}\right\}, & a > 0
\end{aligned}
\]

\noindent The equivariant ansatz and orbit invariant are

\begin{equation*}
        F=-\log y-\log w-2\log z+\phi(t),
        \qquad
        t=\log\frac yz-\frac xz+a\log\frac wz,
\end{equation*}
The Hessian is
\[
\resizebox{\textwidth}{!}{$\displaystyle
\frac{1}{y^2w^2z^4}
\begin{vmatrix}
\ddot\phi & -\ddot\phi\frac{z}{y} & \dot\phi - \ddot\phi(\frac{x}{z} - (a + 1)) & -a\ddot\phi\frac{z}{w}\\
-\ddot\phi\frac{y}{z} & \ddot\phi + 1 - \dot\phi & \ddot\phi(\frac{xy}{z^2} - \frac{(a+1)y}{z}) & a\ddot\phi\frac{y}{w}\\
\dot\phi - \ddot\phi(\frac{x}{z} - (a + 1)) & \ddot\phi(\frac{x}{y} - \frac{z(a+1)}{y}) & \ddot\phi\left( \frac{x}{z} - (a + 1) \right)^2 + 2 + \dot\phi(a + 1 - 2\frac{x}{z}) & a\ddot\phi(\frac{x}{w} - (a+1)\frac{z}{w})\\
-a\ddot\phi\frac{w}{z} & a\ddot\phi\frac{w}{y} & a\ddot\phi(\frac{xw}{z^2} - (a + 1)\frac{w}{z}) & a^2\ddot\phi + 1 - a\dot\phi
\end{vmatrix}
$}
\]

\noindent Equation \eqref{eq:ma} reduces to
\begin{equation}\label{eq:exppowode}
        -a\phid^4+a\phid^3+4a\phid^2\phidd-3a\phid\phidd
        +\phid^3-\phid^2-3\phid\phidd+2\phidd=e^{2\phi}.
\end{equation}

Let $p = \dot\phi$. Equivalently,
\begin{equation*}
        M(p)\dot p+N(p)=e^{2\phi},
\end{equation*}
where
\begin{equation*}
        M(p)=4ap^2-3(a+1)p+2,
        \qquad
        N(p)=-ap^4+(a+1)p^3-p^2.
\end{equation*}
The cancellation
\[
        N'(p)+pM(p)=0
\]
gives
\[
\begin{aligned}
\frac{d}{dt}
\left(e^\phi+N(p)e^{-\phi}\right)
&=
e^{-\phi}
\left(
p e^{2\phi}+N'(p)\dot p-pN(p)
\right)\\
&=
e^{-\phi}
\left(
p(M\dot p+N)-pM\dot p-pN
\right)=0.
\end{aligned}
\]
\begin{equation}\label{eq:exppowfirst}
        e^\phi+N(p)e^{-\phi}=C.
\end{equation}

The condition for positive definiteness of the Hessian is written as follows:

\[
\begin{cases}
\ddot\phi > 0,\\[2mm]
\dot\phi < 1,\\[2mm]
a\dot\phi < 1,\\[2mm]
M\ddot\phi + N > 0.
\end{cases}
\]

\noindent Then $-N = \dot\phi^2(\dot\phi - 1)(a\dot\phi - 1) > 0$ and

\begin{equation}\label{exppower_e^phi}
e^\phi = \dfrac{c}{2} + \sqrt{-N + \dfrac{c^2}{4}}
\end{equation}

\noindent From \eqref{eq:exppowode} we obtain

\[
\dfrac{dt}{d\dot\phi} = \dfrac{1}{\ddot\phi} = \dfrac{M}{e^{2\phi} - N} = \dfrac{M}{2\sqrt{-N + \dfrac{c^2}{4}}\left(\dfrac{c}{2} +\sqrt{-N +\dfrac{c^2}{4}}\right)}
\]

\noindent This integrates to

\[ t = \int_{0}^{\bkappa} \dfrac{4a + 3(a+1)\bkappa + 2\bkappa^2}{2\sqrt{a + (a+1)\bkappa + \bkappa^2 + \dfrac{c^2}{4}\bkappa^4}\left(\dfrac{c}{2}\bkappa^2 + \sqrt{a + (a+1)\bkappa + \bkappa^2 + \dfrac{c^2}{4}\bkappa^4}\right)}d\bkappa \]

\noindent Using \eqref{exppower_e^phi}, we can prove that $c \leq 0$: affine half-line $\left\{ (x, 1, 1, 1) \mid x \leq 0 \right\} \subset K$ intersects all level surfaces of $F$. That yields:

\begin{align*}
\lim_{x\to-\infty} F(x,1,1,1)=-\infty \implies
\lim_{t\to+\infty} \phi(t) = -\infty \implies
\lim_{t\to+\infty} e^{\phi} = 0 \implies
c \leq 0
\end{align*}

As in the previous subsection, let us consider two cases: in the first, we arrive at a contradiction, while in the second, we obtain a parameterized expression for the function.

\textbf{Case 1:} $c < 0$. The integrand above has the asymptotic expansion

\[
2+\dfrac{a+1}{\bkappa}+O(\bkappa^{-2}),\qquad \bkappa\to+\infty.
\]

\noindent This integrates to

\[
t = 2\bkappa + (a + 1)\log\bkappa + O(1) \implies \bkappa = \dfrac{t}{2} - \dfrac{a + 1}{2}\log t + O(1)
\]

\noindent We obtain the asymptotics of $\phi$ using \eqref{exppower_e^phi}:

\[
\phi = \log\left( -\dfrac{4}{c} \right) - 2\log t + O\left( \dfrac{\log t}{t} \right)
\]

\noindent Therefore 

\[
\begin{aligned}
\lim_{z\to0^+}F(-1,1,z,1)
&=\lim_{z\to0^+}\left(-2\log z+
\phi\left(\frac1z-(a+1)\log z\right)\right) \\
&=\log\left(-\frac{4}{c}\right).
\end{aligned}
\]

\noindent Again, this contradicts the boundary condition $F|_{\partial K}=+\infty$. Hence the complete solution cannot correspond to $c<0$.

\textbf{Case 2:} $c = 0$. The integrand simplifies, and it follows that

\[t = \bkappa + \dfrac{a}{2}\log(a+\bkappa) + \dfrac{1}{2}\log(1 + \bkappa) - \dfrac{a}{2}\log a\]

\noindent Finally, \eqref{exppower_e^phi} yields the parametric representation for the solution $\phi(t)$, where $\bkappa$ runs through all positive reals:

\[
\begin{pmatrix}
t \\
\phi
\end{pmatrix}
=
\frac12
\begin{pmatrix}
\log(1+\bkappa)+a\log(a+\bkappa)+2\bkappa - a\log a\\
\log(1+\bkappa)+\log(a + \bkappa) - 4\log\bkappa
\end{pmatrix}.
\]

\subsection{Nonabelian Models}

For the asymmetric mixed-power cones, let
\begin{equation*}
t=\frac{w}{\Delta^{\gamma/2}z^{1-\gamma}},\qquad
F=-\frac{3+\gamma}{2}\log\Delta-(1-\gamma)\log z+\phi(t).
\end{equation*}
Evaluation at the normal section
$(x,y,z,w)^T=(1/2,0,1,t)^T$ and direct determinant computation give
\begin{equation}\label{eq:lorpowerode}
\begin{split}
&(\gamma t\dot\phi+\gamma+3)
\big[\gamma(\gamma-2)t\dot\phi^3
-2(\gamma^2-2\gamma+3)\dot\phi^2\\
&\hspace{29mm}
-4\gamma(\gamma-2)t\dot\phi\ddot\phi
+(15+2\gamma-\gamma^2)\ddot\phi\big]=e^{2\phi}.
\end{split}
\end{equation}
Thus the canonical barrier is reduced to the following two-point nonlinear
boundary-value problem:
\begin{equation}\label{eq:lorpower-bvp}
\begin{gathered}
t\in(-\alpha,\beta),\qquad \text{\eqref{eq:lorpowerode} holds},
\qquad F''\succ0,\\
\phi(t)\longrightarrow+\infty\quad\text{as }t\downarrow-\alpha
\quad\text{and as }t\uparrow\beta.
\end{gathered}
\end{equation}
Here $F''\succ0$ refers to the Hessian of the displayed reconstructed
four-variable function $F$. On the normal section, $t=-\alpha$ and $t=\beta$
correspond to the two lateral boundary hypersurfaces, along which $F_0$
remains finite. The remaining boundary strata, including $\Delta=0$ and
$z=0$, are controlled by the logarithmic terms
$F_0=-\frac{3+\gamma}{2}\log\Delta-(1-\gamma)\log z$.
Hence the lateral boundary condition is equivalent to the two stated blow-up
conditions for $\phi$.

For the logarithmic-parabolic cone, use
\begin{equation*}
t=\frac wz-\log\frac{\Delta}{z^2}<0,\qquad
F=-\frac32\log\Delta-\log z+\phi(t).
\end{equation*}
At the same normal section the equation becomes
\begin{equation}\label{eq:logode}
(2\dot\phi+3)
\left[(16\dot\phi+15)\ddot\phi-4\dot\phi^3-6\dot\phi^2\right]
=e^{2\phi}.
\end{equation}
The corresponding boundary-value problem is
\begin{equation}\label{eq:log-bvp}
\begin{gathered}
t\in(-\infty,0),\qquad \text{\eqref{eq:logode} holds},
\qquad F''\succ0,\\
\phi(t)\to+\infty\quad(t\uparrow0),\qquad
\phi(t)+4\log(-t)\to+\infty\quad(t\to-\infty).
\end{gathered}
\end{equation}
Indeed, the second condition follows from logarithmic homogeneity because
$(-t)^{-1}(1/2,0,1,t)^T$ tends to the boundary ray
$(0,0,0,-1)^T$. Equations \eqref{eq:lorpower-bvp} and \eqref{eq:log-bvp}
are the final outputs of the two nonabelian reductions.

This form is also convenient computationally. One may solve the equations on
truncated intervals, impose the asymptotic boundary data at their endpoints,
and monitor $F''\succ0$ along the normal section. Logarithmic homogeneity and
the explicit automorphisms then reconstruct the barrier on the full cone from
the one-dimensional numerical solution.

\section{Conclusion} \label{sec:Conclusion}

We constructed six explicit cone families in $\mathbb R^4$ and a complete set of Lie algebra generators for a three-dimensional subgroup $G\subset\Aut(K)$ in each case. Knowledge of the behaviour of the canonical barrier on generic $G$-orbits permit to reduce the Monge-Amp\`ere equation \eqref{eq:ma} to a scalar ODE. The homogeneous exceptions are explicit: $\lambda=\tfrac12$ gives $\mathcal L_4$, while zero-weight power degenerations give products $L_3 \times \mathbb R_+$ or $\mathbb R_+^4$.

For the abelian power models we obtained first integrals and parametric candidate solutions of the ODE, together with the required positivity, matching, and boundary conditions. General uniqueness of the integration constants is not claimed. The exponential models remain explicit, and the nonabelian models give nonlinear boundary-value problems. The Hessian-based reductions, obtained by direct determinant computation, provide practical starting points for a numerical barrier construction.



\end{document}